\documentclass[12pt,reqno]{amsart}
\usepackage{amsmath}
\usepackage{amsfonts}
\usepackage{amssymb,esint}
\usepackage{mathrsfs}
\usepackage{hyperref}
\usepackage{stmaryrd}
\usepackage{multirow}
\usepackage{verbatim}
\usepackage{cancel}

\usepackage[T1]{fontenc}
\usepackage[utf8]{inputenc}
\usepackage[english]{babel}
\usepackage[a4paper,top=2cm,bottom=2cm,left=2cm,right=2cm]{geometry}
\usepackage{graphicx}
\usepackage{amsthm}
\usepackage{braket}
\usepackage[autostyle]{csquotes}
\usepackage{pdfpages}
\hypersetup{
	colorlinks=true,
	citecolor=black,
	linkcolor=black,}
\usepackage{cleveref}
\usepackage{tikz-cd}
\usepackage{enumerate}
\usepackage{fancyhdr}
\usepackage{array}
\usepackage{mathabx}
\usepackage{xspace}

\numberwithin{equation}{section}

\theoremstyle{plain}
\newtheorem{theorem}{Theorem}[section]
\newtheorem{lemma}[theorem]{Lemma}

\newtheorem{proposition}[theorem]{Proposition}
\theoremstyle{definition}
\newtheorem{definition}[theorem]{Definition}
\newtheorem{remark}[theorem]{Remark}
\newtheorem{example}[theorem]{Example}

\newcommand{\R}{\mathbb R}
\newcommand{\N}{\mathbb N}
\renewcommand{\H}{\mathbb H}
\newcommand{\G}{\mathbb G}

\newcommand{\galg}{\mathfrak g}

\newcommand{\spn}{\mathrm{span}}
\newcommand{\rank}{\mathrm{rank}}

\newcommand{\abn}{\mathrm{Abn}}
\newcommand{\F}{\mathbb{F}}

\newcommand{\sk}{\mathrm{Skew}}
\newcommand{\de}{d}
\newcommand{\supp}{\mathrm{spt}}
\newcommand{\gr}{\mathrm{Gr}}
\newcommand{\g}{\mathfrak{g}}
\newcommand{\f}{\mathfrak{f}}
\newcommand{\akw}{\mathcal{A}_{k,W}}
\newcommand{\akwm}{\mathcal{A}_{k,W}^m}
\newcommand{\bigzero}{\mbox{\normalfont\Large\bfseries 0}}

\newcommand{\mfk}{\mathfrak}

\newcommand{\IM}{\mathrm{Im} \,}
\DeclareMathOperator*{\Span}{span}

\author{Francesco Boarotto}
\address{Dipartimento di Matematica ``T. Levi-Civita'', via Trieste 63, 35121 Padova, Italy.}
\email{francesco.boarotto@gmail.com}

\author{Luca Nalon}
\address{D\'epartement de Math\'ematiques, Ch. du mus\'ee 23, 1700 Fribourg (CH)}
\email{luca.nalon@unifr.ch}
\thanks{L.N. is partially supported by the Swiss National Science Foundation
	(grant 200021-204501 \emph{Regularity of sub-Riemannian geodesics and
		applications})
	and by the European Research Council (ERC Starting Grant 713998 GeoMeG
	\emph{Geometry of Metric Groups}).}

\author{Davide Vittone}
\address{Dipartimento di Matematica ``T. Levi-Civita'', via Trieste 63, 35121 Padova, Italy.}
\email{davide.vittone@unipd.it}
\thanks{D.V. is supported by University of Padova and by GNAMPA of INdAM (Italy).}



\begin{document}


\title{The Sard problem in step 2 and in filiform Carnot groups}

\begin{abstract}
We study the Sard problem for the endpoint map in some well-known classes of Carnot groups. Our first main result deals with step 2 Carnot groups, where we provide lower bounds (depending only on the algebra of the group) on the codimension of the abnormal set; it turns out that our bound is always at least 3, which improves the result proved in~\cite{LDMOPV} and settles  a question emerged in~\cite{OV}. In our second main result we characterize the abnormal set in filiform groups and show that it is either a horizontal line, or a 3-dimensional algebraic variety.
\end{abstract}

\maketitle

\section{Introduction}
One of the main open questions in sub-Riemannian geometry is arguably the {\em Sard problem} for the endpoint map, cf. \cite{AgrSomeOpen} or \cite[Section 10.2]{Montgomery}: in fact, the  problem is ubiquitous in sub-Riemannian geometry, as it has implications on the regularity of geodesics, the regularity of the distance and of its spheres, the heat diffusion, the analytic-hypoellipticity of sub-Laplacians, etc.
The Sard problem asks whether the {\em abnormal set}, i.e. the set of critical values of the {\em endpoint map}, is negligible or not; we refer to Section~\ref{sec_prelim} for precise definitions. Despite such a simple formulation, only very partial results are 
known~\cite{AgraAny,BelottoFPR,BelottoRifford,RiffTre,ZZ}
even in  settings with a rich structure such as {\em Carnot groups}~\cite{Gent_Hor,BV,CJT06,LLMV_GAFA,LLMV_CAG,LDMOPV,OV}.
The goal of this note is  to provide a contribution in two meaningful classes of Carnot groups: those with nilpotency step 2, and filiform ones.

The Sard problem in step 2 Carnot groups has already been answered affirmatively in~\cite{LDMOPV}; however, the question of getting  finer estimates on the size of the abnormal set was left open. In fact, in~\cite{OV} it was conjectured that the abnormal set $\abn_\G$ in a Carnot group $\G$ of step 2 has codimension at least 3. The conjecture is true in {\em free} Carnot groups of step 2, as proved in~\cite{LDMOPV}, and in some classes of step 2 Carnot groups~\cite{OV} including all groups of topological dimension up to 7. Our first main result, Theorem~\ref{thm_step2} below, is a full, positive answer to this question: we prove in fact that $\abn_\G$ has codimension at least 3 in every step 2 Carnot group $\G$. Actually, our proof also provides,  for the codimension of $\abn_\G$, a lower bound that is purely algebraic, depending only on the   algebraic structure of $\G$, and which could possibly be greater than 3. 

Recall that the stratified Lie algebra $\g=\g_1\oplus\g_2$ of $\G$ can always be seen as the quotient of a free stratified Lie algebra $\mfk f_r$ of step 2 and rank $r:=\dim\g_1$. In turn, $\mfk f_r$ can be identified with $V\oplus\bigwedge^2V$, where the vector space $V$ is $V:=\g_1$ and $\bigwedge^2V$ is the $r(r-1)/2$-dimensional space of 2-vectors on $V$. The second layer $\g_2$ can be identified with the quotient $\bigwedge^2 V/W$ by a vector subspace $W$ of $\bigwedge^2V$; eventually, we identify $\G$ and $\g$ by the exponential map. With this notation we can state our first main result.

\begin{theorem}\label{thm_step2}
	Let $\G$ be a Carnot group of step $2$ given by
	\begin{equation*}
		\G = V \oplus \frac{\bigwedge^2 V}{W}  
	\end{equation*}
for some $W \le \bigwedge^2 V$. Let $W^{\perp_2}$ be the orthogonal to $W$ in $\bigwedge^2 V$ with respect to an adequate scalar product and let $\tilde{k} := \min \set{ \rank(\omega) \, | \, \omega \in W^{\perp_2} \setminus \set{0}}$. Then $\abn_\G$ is contained in an algebraic variety of codimension $2\tilde{k}+1$. In particular, the abnormal set $\abn_\G$ has codimension at least $3$.
\end{theorem}

We refer to Proposition~\ref{prop_adeguato} for the definition of {\em adequate} scalar product. 
The bound $2\tilde k+1$ on the codimension of $\abn_\G$ is not always optimal, see Examples~\ref{exam_1} and~\ref{exam_2}; however, in some cases it provides the exact estimate, as  for instance in the well-known case of {\em Heisenberg groups} (see Example~\ref{exam_Heis}) where the abnormal set is the singleton $\{0\}$ and the optimal bound is indeed greater than 3.

The idea behind the proof of Theorem~\ref{thm_step2} can be described as follows. There is a natural projection map $\pi:\mathbb F_r\to\G$ from the free Carnot group $\F_r$ associated with $\mfk f_r$ onto $\G$. Each horizontal curve  from the identity  in $\G$ can be uniquely lifted to a horizontal curve  from the identity in $\mathbb F_r$. Abnormal curves in $\G$ are projections of abnormal curves  in $\mathbb F_r$, but the converse is not always true: in step 2 Carnot groups one is able to characterize those abnormal curves in $\mathbb F_r$ that project to abnormal curves in $\G$ and this leads to the precise description of $\abn_\G$ contained in formula~\eqref{eq_abnGProp2.5}  (see also~\cite[Proposition~2.5]{OV}). In Proposition~\ref{prop_EkW} we use this description to study $\abn_\G$ and eventually prove Theorem~\ref{thm_step2}.

Our second main result settles the Sard problem in another well-studied class of Carnot groups, the one of {\em filiform}  groups, where the question was left open. Filiform groups are Carnot groups whose stratified Lie algebra $\g=\g_1\oplus\dots\oplus\g_s$ satisfies 
\[
\dim\g_1=2,\qquad\dim \g_2=\dots=\dim \g_s=1.
\]
In particular, $\dim\G=s+1$. 
As proved in~\cite{Vergne}, filiform groups fall into two subclasses: {\em type I} filiform groups and {\em type II} filiform groups. While type I filiform groups can be of any step $s\geq 2$,  type II ones  always have an odd nilpotency step; since  type I and type II filiform groups are isomorphic when $s=3$ (both coinciding with the well-known {\em Engel groups}), we adopt the convention that  $s\geq 5$ for type II filiform groups. We refer to \S~\ref{subsec_filiformi} for precise definitions and we now pass to the statement of our second main result.

\begin{theorem}\label{thm_filiformi}
Let $\G$ be a filiform group of step $s\geq 3$.

\begin{itemize}
\item[(i)] If $\G$ is a type I filiform group, then $\abn_\G$ is a horizontal line.
\item[(ii)] If $\G$ is a type II filiform group, then $\abn_\G$ is an algebraic variety of dimension 3.
\end{itemize}
\end{theorem}

Theorem~\ref{thm_filiformi} follows from Propositions~\ref{prop_filiformiI} and~\ref{prop_filiformiII} together with Remark~\ref{rem_struttura_abnormali_filiformiII}; in these results the singular controls and the associated abnormal curves are also  characterized.
The statement of Theorem~\ref{thm_filiformi} does not include the non-interesting cases $s=1$, when $\G=\R^2$ and $\abn_\G=\emptyset$, and $s=2$, when $\G$ is the first Heisenberg group.

Theorems~\ref{thm_step2} and~\ref{thm_filiformi} show that  the codimension of the abnormal set is at least 3 both in step 2 and in filiform Carnot groups; actually, to our best knowledge there is currently no example of a Carnot group (nor of an {\em equiregular} sub-Riemannian manifold) where the abnormal set has codimension less than 3. This might lead to formulate the following ``strong'' Sard conjecture: is it true that, in Carnot groups and/or in equiregular manifolds, the codimension of the abnormal set  is at least 3? Observe that there are counterexamples when the equiregularity assumption is not in force, the most noticeable being the Martinet structure in $\R^3$ \cite{Mar70}, where the abnormal set is contained in a plane and has therefore codimension 1. At any rate, the answer to this question seems for the moment out of reach.\medskip

{\em Acknowledgments.} We are grateful to E.~Le~Donne for suggesting us to study the Sard problem in filiform groups as well as for several interesting discussions.

\section{Preliminaries}\label{sec_prelim}
A {\em Carnot group} $\G$ of rank $r$ and step $s$ is a connected, simply connected and nilpotent Lie group whose Lie algebra $\mfk{g}$, here identified with the tangent at the group identity $e$, admits a stratification of the form:
	\[
		 \mfk{g}=\mfk{g}_1\oplus \dots \oplus \mfk{g}_s,
	\]
	with $\mfk{g}_{i+1}=[\mfk{g},\mfk{g}_i]$ for $1\le i\le s-1$, $[\mfk g,\mfk g_s]=\{0\}$ and $\dim(\mfk{g}_1)=r$. 
The exponential map $\exp:\galg\to\G$ is a diffeomorphism.

	Denoting by $L_g$ the left-translation on $\G$ by an element $g\in\G$, we consider the {\em endpoint map}
	\[
		\begin{aligned}
			F: L^1([0,1],\mfk{g}_1)&\to \G,\\
				u&\mapsto \gamma_u(1),
		\end{aligned}
	\]
	where we denoted by $\gamma_u:[0,1]\to \G$ the absolutely continuous curve issuing from  $e$, whose derivative is given by $(dL_{\gamma_u(t)})_eu(t)$ for a.e. $t\in [0,1]$. Any such curve $\gamma_u$ is called {\em horizontal}.

The following  Proposition~\ref{prop_prop11BV} was proved in \cite[Proposition~11]{BV}. Let us state some notation. We fix a basis $X_1,\dots,X_r$ of $\mfk g_1$, so that we can identify $\mfk g_1\equiv\R^r$ by $\R^r\ni u\leftrightarrows X_u\in\mfk g_1$ where
\[
X_u:=u_1X_1+\dots+u_rX_r.
\]
As customary, for $X,Y\in\mfk g$ we write $\mathrm{ad}_X(Y):=[X,Y]$, while for $p\in\G$ we denote by $R_p:\G\to\G$ the associated right-translation $R_p(q)=qp$. Eventually, given $t\geq0$ and an integer $j\geq 1$ we introduce the $j$-dimensional simplex $\Sigma_j(t)$ of side $t$ by
\[
\Sigma_j(t):=\{(\tau_1,\dots,\tau_j)\in\R^j\mid 0\le\tau_j\le\dots\le\tau_1\le t\}
\]
With this notation one has the following.

\begin{proposition}\label{prop_prop11BV}
Let $u\in L^1([0,1],\galg_1)$ be a control; then 
\begin{equation}\label{eq_alt_char_im}
\IM(d_u F)=d_eR_{\gamma_u(1)}\left(\Span_{Y\in \mfk{g}_1,t\in [0,1]}\left\{ Y+\sum_{j=1}^{s-1}\int_{\Sigma_j(t)}\left(\mathrm{ad}X_{u(\tau_{j})}\circ\dots\circ  \mathrm{ad}X_{u(\tau_1)}  \right)Y\;d\tau_{j}\dots d\tau_1 \right\}\right).
\end{equation}
In particular
\begin{equation}\label{eq_inv}
d_eR_{\gamma_u(1)}(\mfk g_1)\subset\IM(d_u F).
\end{equation}
\end{proposition}

Statement~\eqref{eq_inv} follows from~\eqref{eq_alt_char_im} by considering $t=0$, see also~\cite[equation~(2.11)]{BV}. Proposition~\ref{prop_prop11BV} is crucial for the study of abnormal curves, that we now introduce.

\begin{definition}\label{defi:abnormal_set}
Let $\G$ be a Carnot group; we say that $u\in L^1([0,1],\galg_1)$ is a {\em singular control} if the differential $d_uF$ of the endpoint map at $u$ is not surjective.
A horizontal curve $\gamma_u$ from the identity $e$ is called {\em singular} (or {\em abnormal}) {\em curve} if the associated control $u$ is singular.
The {\em abnormal set} $\abn_\G\subset\G$ is the set of critical values of $F$, i.e., 
\[
\abn_\G:=\{\gamma_u(1)\mid u\in L^1([0,1],\galg_1)\text{ is a singular control}\}.
\]
In particular, a point $g\in \G$ belongs to $\abn_\G$ if and only if there exists an abnormal  curve  joining $e$ and $g$.
\end{definition}

In the setting of Carnot groups of step 2, Proposition~\ref{prop_prop11BV} allows to characterize abnormal curves in a particularly simple way. Let $\G$ be a fixed Carnot group of step 2 with Lie algebra $\galg=\galg_1 \oplus \galg_2$ and denote by $\pi_1: \galg_1 \oplus \galg_2 \to \galg_1$ the canonical projection onto the first layer. Given an horizontal curve $\gamma$ in $\G$ we define
\begin{equation}\label{eq_PgammaIgamma}
\begin{split}
	P_\gamma &= \spn \Set{\pi_1(\gamma(t)) \, | \, t \in [0,1] }\subset\galg_1 \\
	I_\gamma &= \galg_1 \oplus [P_\gamma,\galg_1] \, .
\end{split}
\end{equation}	
Starting from Proposition~\ref{prop_prop11BV} it is not difficult to realize that $\IM(d_u F)=d_eR_{\gamma_u(1)}(I_\gamma)$; in particular,
	\begin{equation}\label{eq_abnormalIFF}
		\text{$\gamma$ abnormal} 
		\quad \Leftrightarrow \quad I_\gamma \neq \galg 
		\quad \Leftrightarrow \quad  [P_\gamma,\galg_1]\neq\galg_2 \, .
	\end{equation}
For more details see~\cite[Section~2]{OV}. 

The characterization~\eqref{eq_abnormalIFF} of abnormal curves, in turn, allows to find an explicit, purely algebraic formula (see~\eqref{eq_abnGProp2.5} below) for the abnormal set $\abn_\G$: we however postpone its proof in order to fist settle some notation and preliminary material about step 2 Carnot groups.


\section{The Sard problem in step 2 Carnot groups}
In this section we consider Carnot groups of step $s=2$ and a fixed rank $r\geq 2$. 
Recall that a Carnot group is {\em free} if the only relations imposed on its Lie algebra are those generated by the skew-symmetry and Jacobi’s identity. 
Let us denote by  $\F_r$ the free Carnot group  of step 2 and rank $r$ and by $V=(\mathfrak f_r)_1$ the first layer of the associated Lie algebra $\mathfrak f_r$. Clearly, $\f_r=V \oplus [V,V]$ can be identified with $V \oplus \bigwedge^2V$. 

The following proposition is standard and we omit its proof.

\begin{proposition}\label{prop_adeguato}
		Given a free Lie algebra $\f_r=V \oplus \bigwedge^2V$ and a scalar product on $V$, there is a unique way to extend it to a scalar product on $\f_r$ such that for every  orthonormal basis $\set{e_1,\dots,e_r}$ of $V$, the basis \begin{equation*}
			\set{e_1,\dots,e_r} \cup \Set{ e_i \wedge e_j \, | \, i<j}
		\end{equation*}
	is orthonormal. We will refer to such a scalar product as an \emph{adequate scalar product} on $\f_r$.
	\end{proposition}

Let now $\G$ be a fixed Carnot group of step 2 and rank $r$; then its Lie algebra $\galg=\galg_1\oplus\galg_2$ can be seen as the quotient of $\f_r$ through a linear subspace $W \le\bigwedge^2 V $, so that
\begin{equation}\label{eq_quoziente}
\galg_1=V
\qquad\text{and}\qquad 
\galg_2=\frac{\bigwedge^2 V}{W}\,.
\end{equation}
Identifying $\G$ and $\galg$  through the exponential map, we have $\G \cong \pi(\f_r)$ where $\pi$ is the quotient map.

The necessity of using the language of multi-vectors motivates the following subsection, where we  fix some terminology and state some preliminary facts. Most of them are well-known, see however~\cite[\S~I.1]{BCGGG}.

\subsection{Some tools from multi-linear algebra}
Let $V$ be a real vector space of dimension $r$. We will refer to
	\begin{equation*}
		\gr(k,V)= \Set{ W \le V \, | \, \dim W = k} 
	\end{equation*}
	as the \emph{rank $k$ Grassmanian of $V$}. It has a structure of real projective algebraic variety, with no singular points. 
	Let us consider the following open subset of $V^k$: \begin{equation*}
		I_k = \Set{ (v_1, \dots, v_k) \in V^k \, | \, \text{$v_1, \dots, v_k$ are linearly independent}} \, ,
	\end{equation*}
	then $\gr(k,V)$ also has a structure of smooth manifold such that the map
	\begin{equation*}
		\spn \colon I_k \to \gr(k,V) \qquad (v_1,\dots,v_k) \mapsto \spn\set{v_1,\dots,v_k}
	\end{equation*}
is smooth. The topology of such manifold is equivalent to the topology of $\gr(k,V)$ as an algebraic variety (with the topology given by the inclusion).
	
%

Given a basis $e_1,\dots,e_r$ of $V$, the second exterior power $\bigwedge^2 V$ of $V$  is generated by $\set{e_i \wedge e_j \, | \, i < j \,}$. 
The {\em rank} of $\omega\in \bigwedge^2 V$ is the smallest integer $k$ such that one can find $v_1,\dots,v_{2k}$ for which $\omega=v_1\wedge v_2+\dots+v_{2k-1}\wedge v_{2k}$; in this case we write $\rank(\omega)=k$ and define the {\em support} (or {\em span}, see also~\cite[Section~5]{AlbertiMarchese}) of $\omega$
\[
\supp(\omega):=\Span\{v_1,\dots,v_{2k}\}.
\]
As a matter of fact, $\supp(\omega) $ is the smallest subspace $W\leq V$ such that $\omega\in \bigwedge^2 W$, hence it is independent  of the choice of $v_1,\dots,v_{2k}$.

Once a basis for $V$ is fixed, the space $\bigwedge^2 V$ can be canonically identified with the space $\sk(r,\R)$ of skew-symmetric $r \times r$ matrices by
	\begin{equation*}
		A \quad \rightleftarrows \quad \omega(A)=\sum_{i<j} A^i_j (e_i \wedge e_j).
	\end{equation*}
Observe that the rank of $A$ (as a matrix) is twice the rank of  $\omega(A)$ (as a 2-vector).
We define 
	\begin{align*}
		\left(\bigwedge\nolimits^{\!2} V\right)_k &:= \Set{ \omega \in \bigwedge\nolimits^{\!2} V \, | \, \rank(\omega)=k \, } \\
		\left(\bigwedge\nolimits^{\!2} V\right)_{\le k} &:= \Set{ \omega \in \bigwedge\nolimits^{\!2} V \, | \, \rank(\omega)\le k \, } \\
		\mathcal{A}_k &:= \Set{ A \in \sk(r,\R) \, | \, \rank(A)=2k \, } \\
		\mathcal{A}_{\le k} &:= \Set{ A \in \sk(r,\R) \, | \, \rank(A)\le 2k \, }
			\end{align*}
and we write down the following identifications:
	\begin{equation}\label{eq_matrici2vettori}
		\left(\bigwedge\nolimits^{\!2} V\right)_k \leftrightarrows \mathcal{A}_k \qquad \left(\bigwedge\nolimits^{\!2} V\right)_{\le k} \leftrightarrows \mathcal{A}_{\le k} \, .
	\end{equation}	
	We observe that $\mathcal{A}_{\le k}$ is an affine algebraic variety, as we impose every minor of order $2k+1$ to be zero, while $\mathcal{A}_k$ is an affine semi-algebraic variety, as we impose every minor of order $2k+1$ to be zero and the sum of all squared minors of order $2k$ to be positive, so that at least one is non-zero.
	
	We will now prove that $\mathcal{A}_k$, endowed with the topology induced by the inclusion in $\sk(r,\R)$, has a structure of smooth manifold of dimension $k(2r-2k-1)$. We will also  provide a useful system of local coordinates. The identification  $\sk(r,\R)\equiv\R^{{r(r-1)}/{2}}$ is understood.
	
	\begin{lemma}\label{para}
		Let $\omega = e_1 \wedge e_2 + \cdots + e_{2k-1}\wedge e_{2k}$ be a 2-vector of rank $k$; let  $\set{e_1,\dots,e_{2k}}$ be completed to a basis $\set{e_1,\dots,e_r}$  of $V$ and consider the associated identifications~\eqref{eq_matrici2vettori}. 
		Then there exists an open cone $U\subset \mathcal{A}_k$ containing $\omega$ and rational functions 
		\begin{equation*}
			\lambda_t, \, a^s_t, \, b^s_t \colon 
			U
			\to \R, \qquad 1 \le t \le k \, , \, 2t+1 \le s \le r
		\end{equation*}
	such that 
	\begin{equation*}
		\omega(A)= \sum_{t=1}^k \left( \lambda_t(A)e_{2t-1} + \sum_{s=2t+1}^r a^s_t(A) e_s\right) \wedge \left( e_{2t} + \sum_{s=2t+1}^r b^s_t(A) e_s
		\right) \, .
	\end{equation*}
	In particular, the function $\rho := (\lambda_t, \, a^s_t, \, b^s_t)_{1 \le t \le k, 2t+1 \le s \le r}:U\to\R^{k(2r-2k-1)}$ is an homeomorphism into its image and $\mathcal{A}_k$ has the structure of $k(2r-2k-1)$-dimensional smooth manifold.
	\end{lemma}

\begin{proof}
	Let $A_\omega$ be the matrix associated to $\omega$ with respect to the basis $\set{e_1,\dots,e_r}$. For $A \in \mathcal{A}_k$ we consider the map
	\begin{equation*}
		\omega_0(A)=\sum_{i<j} A_{ij}(e_i \wedge e_j) \, .
	\end{equation*}
	Let us consider $U_1= \set{A \in \mathcal{A}_k, \, A_{12} \neq 0}$, it is an open cone and neighbourhood of $\omega$, for $A \in U_1$ we can write
	\begin{equation*}
	\begin{aligned}
		\omega(A)&= \sum_{i<j} A_{ij}(e_i \wedge e_j)\\ 
		&= \left(\sum_i A_{i2}e_i\right) \wedge \left( e_2 + \sum_{j \ge 3} \frac{A_{1j}}{A_{12}}e_j\right) + \sum_{3 \le i < j} \frac{A_{j2}A_{1i}-A_{i2}A_{1j}}{A_{12}}(e_i \wedge e_j) + \sum_{3 \le i < j} A_{ij}(e_i \wedge e_j) \\
		&=\left(\sum_i A_{i2}e_i\right) \wedge \left( e_2 + \sum_{j \ge 3} \frac{A_{1j}}{A_{12}}e_j\right) + \sum_{3 \le i < j} \left(A_{ij}+ \frac{A_{j2}A_{1i}-A_{i2}A_{1j}}{A_{12}}\right)(e_i \wedge e_j) \, .
	\end{aligned}
	\end{equation*}
We define the rational functions
\[
\lambda_1 := A_{12},\qquad a_1^s := A_{s2},\ b_1^s := \frac{A_{1s}}{A_{12}},\text{ for every } s=3,\dots,r
\]
and let
\[
\omega_1(A):= \sum_{3 \le i < j} \left(A_{ij}+ \frac{A_{j2}A_{1i}-A_{i2}A_{1j}}{A_{12}}\right)(e_i \wedge e_j).
\]
We now notice that $\supp(\omega_1(A)) \subseteq \set{e_3,\dots,e_r}$, moreover $\sum_i A_{i2}e_i$ and $e_2 + \sum_{j \ge 3} \frac{A_{1j}}{A_{12}}e_j$ are jointly linear independent from $\set{e_3,\dots,e_n}$: therefore, $\omega_0(A)$ is the sum of $\omega_1(A)$ and a simple (i.e., with rank 1) vector with support in direct sum with $\supp(\omega_1(A))$, and this implies that $\rank \, \omega_1(A)=\rank \, \omega_0(A)-1=k-1$. 
We now consider the open cone
\[
U_2=\Set{A \in U_1\,\Big|\, A_{34}+ \frac{A_{42}A_{13}-A_{32}A_{14}}{A_{12}}\neq0} \, ,
\]
which is also a neighbourhood of $\omega$, and we recall that rational functions are closed under composition.
	
Iterating this process $k$ times, we  obtain  an open cone $U:=U_k\subset\mathcal{A}_k$ containing $\omega$ and rational functions
	\begin{equation*}
		\lambda_t, \, a^s_t, \, b^s_t
		 \colon 
		 U\to \R, \qquad 1 \le t \le k \, , \, 2t+1 \le s \le r
	\end{equation*}
such that
	\begin{equation*}
		\omega(A)=\sum_{t=1}^k \omega_t(A)
		= \sum_{t=1}^k \left( \lambda_t(A)e_{2t-1} + \sum_{s=2t+1}^r a^s_t(A) e_s\right) \wedge \left( e_{2t} + \sum_{s=2t+1}^r b^s_t(A) 
		e_s\right),  \quad A \in U.
	\end{equation*}
	The vector function $\rho := (\lambda_t, \, a^s_t, \, b^s_t)_{1 \le t \le k, 2t+1 \le s \le r}$ is clearly smooth on $U$, since its components are rational functions. 
The function $\eta : \R^{k(2r-2k-1)} \to \mathcal{A}_k$ defined by 
\[
\begin{aligned}
		(\lambda_t, a_t^s, b_t^s)_{\substack{1 \le t \le k,\hfill\\ 2t+1 \le s \le r}} & \stackrel{\eta}{\longmapsto} \displaystyle\sum_{t=1}^k \Big( \lambda_te_{2t-1} + \sum_{s=2t+1}^r a^s_t e_s\Big) \wedge \Big( e_{2t} + \sum_{s=2t+1}^r b^s_t e_s\Big)  = \sum_{i <j} \eta_{ij} (e_i \wedge e_j)
\end{aligned}
\]
coincides with $\rho^{-1}$ on $\rho(U)$, and its coefficients $\eta_{ij}$ are evidently polynomial functions in $(\lambda_t, a_t^s, b_t^s)$.
Thus $\rho$ is an homeomorphism onto its image and $\mathcal{A}_k$ is a topological manifold. Finally, transition maps are compositions of rational function with a polynomial map, hence smooth. Therefore $\mathcal{A}_k$ is also a smooth manifold.
\end{proof}

\begin{remark}
	In the previous construction, $\lambda_t$ and $a_t^s$ are homogeneous rational functions of degree one while $b_t^s$ are homogeneous rational functions of degree zero. 
\end{remark}

\subsection{Estimate of the abnormal set in Carnot groups of step 2}
Hereafter we will always consider a free Lie algebra $\mfk f_r=V\oplus\bigwedge^2V$ of step 2 and rank $r\geq 2$ equipped with an adequate scalar product, as in Proposition~\ref{prop_adeguato}; the symbol $\perp$ will indicate  orthogonality with respect to such scalar product. Recall that the Lie algebra $\galg=\galg_1 \oplus \galg_2$ of a Carnot group $\G$ of step $2$ and rank $r$ can always be written as in~\eqref{eq_quoziente} for some subspace $W\leq \bigwedge^2V$, that is fixed from now on. Let us introduce the  notation: if $A\le V$ and $B\le\bigwedge^2 V$ (so that $A,B$ are subspaces of $\mfk f_r$ too), then 
\begin{align*}
& A^{\perp_1} := A^\perp \cap V\ \subset\ V\equiv \galg_1\\
& B^{\perp_2} := B^\perp \cap\textstyle \bigwedge^2V\ \subset\ \bigwedge^2V.
\end{align*}

Let us identify $\F_r=\mfk f_r$ and $\G=\galg$ by the associated exponential maps and let $\pi:\F_r=\mfk f_r\to\galg=\G$ the quotient map.
Given a horizontal curve $\gamma_0 \colon [0,1] \to \G$ such that $\gamma_0(0)=0$, there exists a unique horizontal curve ${\gamma} \colon [0,1]\to \F_r$ such that $\gamma_0= \pi \circ {\gamma}$ and ${\gamma}(0)=0$. The curves $\gamma_0$ and ${\gamma}$ are associated with the same control $u(t)$; recalling the notation in~\eqref{eq_PgammaIgamma}, one has $\pi_1(\gamma_0)=\pi_1({\gamma})$, thus $P_{\gamma_0}=P_{{\gamma}}$ and $I_{\gamma_0}=\pi(I_{{\gamma}})$. In particular, an abnormal curve $\gamma \colon [0,1] \to \F_r$ is the lift of an abnormal curve on $\G$ if and only if $\pi(I_\gamma)\neq \galg$, i.e., if and only if $I_\gamma + W \neq V \oplus \bigwedge^2V$, that in turn is equivalent to 
	\begin{equation*}
		I_\gamma^\perp \cap W^\perp \neq \set{0} \, .
	\end{equation*}
	We can compute $I_\gamma^\perp=\left( V \oplus [P_\gamma,V]   \right)^\perp=[P_\gamma,V]^{\perp_2}
=[P_\gamma^{\perp_1},P_\gamma^{\perp_1}]=\textstyle\bigwedge^2 \left( P_\gamma^{\perp_1} \right)$, 
where we used the fact that the subspaces $[P_\gamma,V]$ and $[P_\gamma^{\perp_1},P_\gamma^{\perp_1}]$ are orthogonal and complementary in $\bigwedge^2V$.
Since $\gamma\subset P_\gamma\oplus[P_\gamma,P_\gamma]$, we deduce that
\begin{equation}\label{eq_abnGProp2.5}
	\begin{split}
		\abn_\G &= \pi \left(\bigcup \Set{Q^{\perp_1} \oplus [Q^{\perp_1},Q^{\perp_1}] \, | \, \text{$Q \le V$, ${\textstyle\bigwedge^2} Q \cap W^{\perp_2} \neq \set{0}$}}\right) \\ &= \bigcup \Set{\pi \left(Q^{\perp_1} \oplus \textstyle\bigwedge^2 Q^{\perp_1}\right) \, | \, \text{$Q \le V$, ${\textstyle\bigwedge^2} Q \cap W^{\perp_2} \neq \set{0}$}}\, .
	\end{split}
\end{equation}
This formula for $\abn_\G$ is  equivalent to~\cite[Proposition~2.5]{OV}.
If $g \in \abn_\G$, then $g \in \pi \left(Q^{\perp_1} \oplus \textstyle\bigwedge^2Q^{\perp_1}\right)$ for some $Q \le V$ such that $\bigwedge\nolimits^{\!2} Q \cap W^{\perp_2}\neq \set{0}$, therefore there exists $\omega=\omega(g) \in \bigwedge\nolimits^{\!2} Q \cap W^{\perp_2} \neq \set{0}$.  Thus $\supp(\omega) \le Q$ and $Q^{\perp_1}\le \supp(\omega)^{\perp_1}$. We obtained that
\begin{equation}	\label{eq_abnG_EkW}
	\begin{split}
		\abn_\G 
		&= \bigcup \Set{\pi \left(\supp(\omega)^{\perp_1} \oplus \textstyle\bigwedge^2\supp(\omega)^{\perp_1}\right) \, | \, \omega \in W^{\perp_2} \setminus \set{0}} \\ 
		&= \bigcup_{k}\bigcup \Set{\pi \left(\supp(\omega)^{\perp_1} \oplus \textstyle\bigwedge^2\supp(\omega)^{\perp_1}\right) \, | \, \omega \in W^{\perp_2} \setminus \set{0}  , \, \rank(\omega)=k \, }.
	\end{split}
\end{equation}
	
	The following estimate is the most important result of the present section.

\begin{proposition}\label{prop_EkW}
	The set
	\begin{equation*}
		E_{k,W}:=\bigcup \Set{\pi \left(\supp(\omega)^{\perp_1} \oplus \textstyle\bigwedge^2\supp(\omega)^{\perp_1}\right) \, | \, \omega \in W^{\perp_2} \setminus \set{0}  , \, \rank(\omega)=k \, }
	\end{equation*}
is contained in an algebraic variety of codimension $2k+1$. 
\end{proposition}

\begin{proof}
	We define
	\begin{align*}
		\akw&=\Set{\omega \in W^{\perp_2} \setminus \set{0} \, | \, \rank(\omega)=k} \, ,\\
		\akwm&=\Set{\omega \in W^{\perp_2} \setminus \set{0} \, | \, \rank(\omega)=k, \, \dim \left(W\cap\textstyle\bigwedge^2\supp(\omega)^{\perp_1} \right)=m} \, , \\
		\mathcal{B}(\omega)&=\pi \left(\supp(\omega)^{\perp_1} \oplus \textstyle\bigwedge^2\supp(\omega)^{\perp_1}\right) \, , \\
		E_{k,W}^m&= \bigcup \Set{\mathcal{B}(\omega) \, | \, \omega \in \akwm \,} \, .
	\end{align*}
We notice that $\akw$ and $\akwm$ are cones and semi-algebraic subvarieties of $\mathcal{A}_k$. Moreover $\mathcal{B}(\lambda\omega)=\mathcal{B}(\omega)$ for $\lambda \in \R \setminus \set{0}$ and
\begin{equation*}
	\dim \mathcal{B}(\omega)= \frac{(r-2k)(r-2k +1)}{2} - m \qquad \text{for $\omega \in \akwm$}
\end{equation*}
that is constant over $\omega \in \akwm$. Therefore $E_{k,W}^m$ is contained in an algebraic variety of dimension 
\begin{equation}\label{eq_ast}
	\frac{(r-2k)(r-2k +1)}{2} - m + R(m,k,W) - 1 \, 
\end{equation}
where $R(m,k,W)$ is the smallest dimension of a smooth semi-algebraic variety containing $\akwm$. Indeed, let $\omega_0 \in \akwm$; $\akwm$ can locally be parametrized by $R(m,k,W) \le \dim \akwm$ parameters around $\omega_0$. Let us consider $\set{ \mathcal{B}(\omega) \, | \, \omega \in \akwm \, }$, that is a semi-algebraic variety (since it is image of a semi-algebraic variety through an algebraic map) whose dimension is not greater than $\dim \akwm - 1$ (since each non-empty fibre of $\set{ \mathcal{B}(\omega) \, | \, \omega \in \akwm \, }$ contains a one-dimensional set as $\mathcal{B}(\lambda\omega)=\mathcal{B}(\omega)$). Finally we recall that the described local parametrization of $\akwm$ also provides an algebraic map $\akwm \to V^{2k}$ that maps $\omega$ to a basis of $\supp(\omega)$, thus we can locally find a basis to $\mathcal{B}(\omega)$ that depends algebraically on $\omega$, using this basis we can parametrize $\mathcal{B}(\omega)$ with $\dim \mathcal{B}(\omega)$ parameters. As a result, for every $x \in E_{k,W}^m$ there is a local algebraic submersion from a neighbourhood of $\R^{
R(m,k,W)}$ to (possibly a superset of) a neighbourhood of $x$.

The final step of the proof is to estimate $R(m,k,W)$: we will prove that 
\[
R(m,k,W) \le \dim \mathcal{A}_k - (n-m) = k(2r-2k-1) - n +m \,,
\]
where we used Lemma~\ref{para}. 
Let $\omega_0 \in \akwm$ be fixed, then $\omega_0 = e_1 \wedge e_2 + \cdots + e_{2k-1} \wedge e_{2k}$ for linearly independent vectors $e_1,\dots,e_{2k} \in V$, non necessarily orthogonal. 
Let us consider  an orthonormal basis $\set{e_{2k+1},\dots,e_{r}}$ of $\supp(\omega_0)^{\perp_1}$ so that $\set{e_1,\dots,e_{2k},e_{2k+1},\dots,e_r}$ is a basis for $V$. Considering this basis, Lemma~\ref{para} provides the local (around $\omega_0$) parametrization of $\mathcal{A}_k\equiv(\bigwedge^2V)_k$ (which is a superset of $\akwm$)  given by
\begin{equation*}
	\omega(\xi,a,b)= \omega_0+\sum_{t=1}^k \bigg[\xi_t(e_{2t-1} \wedge e_{2t}
	 +  \sum_{s=2t+1}^r \Big( a^s_t(e_{2t} \wedge e_s) +  b^s_t (e_{2t-1} \wedge e_s)\Big)\bigg] + \sum_{t < s} Q_{ts}(\xi,a,b)(e_t \wedge e_s)
\end{equation*}
where $Q_{ts}(\xi,a,b)$ are homogeneous polynomial of degree $2$ in $(\xi,a,b)$. 
In the previous parametrization we used Lemma~\ref{para} making the useful substitution $\lambda_i = 1 + \xi_i$, so that $\omega_0=\omega(0,0,0)$. Let us now consider  the basis $v_1,\dots,v_r$ of $V$ defined by
\begin{equation*}
	v_i \in \left(\spn \set{ e_j \, | \, j \neq i \, }\right)^\perp \qquad \braket{e_i,v_i}=1  \quad \forall \, i=1,\dots,r \, ,
\end{equation*} 
so that
\[
\begin{array}{l}
\braket{v_i,e_j}=\delta_{ij}\qquad \text{for every }i,j=1,\dots,r\\
v_j=e_j\qquad  \text{for every }j=2k+1,\dots,r\\
\spn\set{v_1,\dots,v_{2k}}=\spn \set{ e_1,\dots,e_{2k}}
\end{array}
\]
Setting $n:=\dim W$, let $\theta^1,\dots,\theta^n$ be a set of generators for $W$  such that
\begin{equation}\label{eq_bullet}
\spn\set{\theta^{n-m+1},\dots,\theta^n}=W\cap \textstyle\bigwedge^2\supp(\omega_0)^{\perp_1}
=W\cap \textstyle\bigwedge^2\spn\set{v_{2k+1},\dots,v_r}.
\end{equation}
Let $\theta^h_{ij}$ be the coordinates of $\theta^h$ with respect to the basis $\set{ v_i \wedge v_j \, | \, i < j}$, i.e.,
\begin{equation*}
	\theta^h = \sum_{i<j} \theta^h_{ij}(v_i \wedge v_j), \qquad h=1,\dots,n \, .
\end{equation*}
By~\eqref{eq_bullet} 
\begin{equation}\label{eq_star}
\theta^h_{ij}=0\qquad\text{whenever }n-m+1\le h\le n,\ i\le 2k\text{ and }i<j.
\end{equation}
The orthogonality between $\omega(\xi,a,b)$ and $W$ can be expressed by the  system
\begin{equation*}
\braket{\omega(\xi,a,b), \theta^1}=\dots=\braket{\omega(\xi,a,b), \theta^n}=0
\end{equation*}
which, due to 
\begin{align*}
\braket{v_i \wedge v_j, e_h \wedge e_\ell}
&=\braket{v_i,e_h}\braket{v_j,e_\ell} - \braket{v_i,e_\ell}\braket{v_j,e_h} \\
&=\delta_{ih}\delta_{j\ell}-\delta_{i\ell}\delta_{jh}\\
&=\delta_{ih}\delta_{j\ell}\qquad\text{whenever $i<j$ and $h<\ell$,}
\end{align*}
%
in local coordinates becomes
\begin{equation*}
	\begin{cases}
P_1(\xi,a,b):=
\sum_{i=1}^k\Big[ \xi_i\theta^1_{2i-1,2i} +  \sum_{j=2i+1}^r \big(a^j_i\theta^1_{2i,j} + b^j_i \theta^1_{2i-1,j}\big) \Big]
+ \sum_{i < j} Q_{ij}(\xi,a,b)\theta^1_{i,j}= 0 \\
		\qquad  \vdots \\
P_n(\xi,a,b):=
\sum_{i=1}^k\Big[ \xi_i\theta^n_{2i-1,2i} +  \sum_{j=2i+1}^r \big(a^j_i\theta^n_{2i,j} + b^j_i \theta^n_{2i-1,j}\big) \Big]
+ \sum_{i < j} Q_{ij}(\xi,a,b)\theta^n_{i,j}= 0 \, .
	\end{cases} 
\end{equation*}
The latter system defines the algebraic variety $\akw$ that contains $\akwm$. We can estimate the dimension of $\akw$ (and therefore the dimension of $\akwm$) computing the rank of $\de P(0,0,0)$ were $P$ is the polynomial function
\begin{equation*}
	P(\xi,a,b)=\begin{pmatrix}
		P_1(\xi,a,b) \\ \vdots \\ P_n(\xi,a,b)
	\end{pmatrix} \, .
\end{equation*}
We observe that $\de \left( \sum_{i < j} Q_{ij}(\xi,a,b)\theta^h_{i,j}\right)(0,0,0)=0$ since $Q_{ij}$ are homogeneous polynomial of degree $2$ in $\xi,a,b$. Now we finally compute
\begin{equation*}
\begin{aligned}
 \de P(0,0,0)&=
 \begin{pmatrix}
 	\theta^1_{1,2} & \cdots & \theta^1_{2k-1,2k} & \left( \theta^1_{2i,j} \right)^{j=2i+1,\dots,r}_{i=1,\dots,k} & \left( \theta^1_{2i-1,j} \right)^{j=2i+1,\dots,r}_{i=1,\dots,k} \\
 	\vdots & \ddots & \vdots & \vdots & \vdots \\
	\theta^n_{1,2} & \cdots & \theta^n_{2k-1,2k} & \left( \theta^n_{2i,j} \right)^{j=2i+1,\dots,r}_{i=1,\dots,k} & \left( \theta^n_{2i-1,j} \right)^{j=2i+1,\dots,r}_{i=1,\dots,k} \\
 \end{pmatrix}
\\
 &= 	\begin{pmatrix} \left( \theta^1_{ij}\right)_{i \le 2k, i<j} \\ \vdots \\ \left( \theta^n_{ij}\right)_{i \le 2k, i<j}
	\end{pmatrix}
 \end{aligned}
\end{equation*}  
so that by~\eqref{eq_star}
\begin{equation*}
	\rank (\de P(0,0,0))= \rank \begin{pmatrix} \left( \theta^1_{ij}\right)_{i \le 2k, i<j} \\ \vdots \\ \left( \theta^n_{ij}\right)_{i \le 2k, i<j}
	\end{pmatrix} = \rank \begin{pmatrix} \left( \theta^1_{ij}\right)_{i \le 2k, i<j} \\ \vdots \\ \left( \theta^{n-m}_{ij}\right)_{i \le 2k, i<j}
\end{pmatrix} \, .
\end{equation*}
Moreover we know that
\begin{equation*}
	\rank \begin{pmatrix} \left( \theta^1_{ij}\right)_{i<j} \\ \vdots \\ \left( \theta^n_{ij}\right)_{i<j}
	\end{pmatrix} =n \qquad \text{and} \qquad \rank \begin{pmatrix} \left( \theta^{n-m+1}_{ij}\right)_{2k+1 \le i<j} \\ \vdots \\ \left( \theta^n_{ij}\right)_{2k+1 \le i<j}
\end{pmatrix} =m,
\end{equation*}
where we used the fact that $\set{\theta^1,\dots,\theta^n}$ is a basis for $W$ and~\eqref{eq_bullet}.
Then
\begin{equation*}
	n= \rank 
	\left(\begin{array}{@{}c|c@{}}
		\begin{matrix}
			\left( \theta^1_{ij}\right)_{i \le 2k, i<j} \\ \vdots \\ \left( \theta^{n-m}_{ij}\right)_{i \le 2k, i<j}			
		\end{matrix}
		& \Asterisk \\
		\hline
		\bigzero &
		\begin{matrix}
			\left( \theta^{n-m+1}_{ij}\right)_{2k+1 \le i<j} \\ \vdots \\ \left( \theta^n_{ij}\right)_{2k+1 \le i<j}
		\end{matrix}
	\end{array}\right) ,
\end{equation*}
hence
\begin{equation*}
n	= \rank \begin{pmatrix} \left( \theta^1_{ij}\right)_{i\le 2k, i<j} \\ \vdots \\ \left( \theta^{n-m}_{ij}\right)_{i\le 2k, i<j}
\end{pmatrix} 
+ \rank \begin{pmatrix} \left( \theta^{n-m+1}_{ij}\right)_{2k+1 \le i<j} \\ \vdots \\ \left( \theta^n_{ij}\right)_{2k+1 \le i<j}
\end{pmatrix} 
\end{equation*}
and
\begin{equation*}
	\rank \left(\de P(0,0,0)\right)= \rank \begin{pmatrix} \left( \theta^1_{ij}\right)_{i \le 2k, i<j} \\ \vdots \\ \left( \theta^{n-m}_{ij}\right)_{i \le 2k, i<j}
	\end{pmatrix} = n - \rank \begin{pmatrix} \left( \theta^{n-m+1}_{ij}\right)_{2k+1 \le i<j} \\ \vdots \\ \left( \theta^n_{ij}\right)_{2k+1 \le i<j}  
\end{pmatrix} = n- m \, .
\end{equation*}
Thus we proved that at any point $\omega_0 \in \akwm$, the algebraic variety $\akwm$ is locally contained in a smooth manifold of dimension
\begin{equation*}
	\dim \mathcal{A}_k - (n-m) = k(2r-2k-1) - n +m \, .
\end{equation*}
Recalling~\eqref{eq_ast}, we can now estimate
\begin{align*}
	\dim E_{k,W}^m&=\frac{(r-2k)(r-2k +1)}{2} - m + R(m,k,W) - 1 \\ &\le \frac{(r-2k)(r-2k +1)}{2} - m + k(2r-2k-1) - n +m - 1 \\
	&= \frac{r(r+1)}{2} -k(2r+1) + 2k^2 + 2kr - 2k^2 - k -1 - n \\
	&= {\frac{r(r+1)}{2} - n} - (2k+1)\\
	&=\dim\G - (2k+1)\, .
\end{align*}
Since $E_{k,W}$ is a finite union of sets $E_{k,W}^m$, we can conclude that also $E_{k,W}$ is contained in an algebraic variety of codimension $2k+1$. 
\end{proof}

We can now prove one of our main results.

\begin{proof}[Proof of Theorem~\ref{thm_step2}]
Recalling~\eqref{eq_abnG_EkW}, we have $\abn_\G = \bigcup_{k} E_{k,W}$
and, by Proposition~\ref{prop_EkW}, each  $E_{k,W}$ has codimension at least $2k+1$ in $\G$. The statement follows by noticing that $\tilde{k} = \min \set{ \rank(\omega) \, | \, \omega \in W^{\perp_2} \setminus \set{0}}$ is the smallest $k$ such that $E_{k,W}$ is non-empty. 
\end{proof}

One may ask whether the estimate on the dimension of the abnormal set provided by Theorem~\ref{thm_step2} is optimal. Heisenberg groups provide a family of Carnot groups where this estimate is optimal. 

\begin{example}\label{exam_Heis}
The \emph{$k$-th Heisenberg group} is the stratified group $\mathbb{H}^k$ whose Lie algebra stratification $\g = \galg_1 \oplus \g_2$ is given by
	\begin{equation*}
		\g_1 = \spn \Set{X_1,\dots,X_k,Y_1,\dots,Y_k} \, , \qquad \g_2 = \spn \Set{T} \, ,
	\end{equation*}
	and the only non-zero Lie brackets between generators are 
	\begin{equation*}
		[X_i,Y_i]= T \qquad \text{for $i=1,\dots,k$} \, .
	\end{equation*}
It can be easily checked that $\g_2=\bigwedge^2 V/W$ for $W:=\spn(X_1\wedge Y_1+\dots+X_k\wedge Y_k)^{\perp_2}$, hence $W^\perp$ is a one dimensional subspace spanned by a $2$-vector of rank $k$. This implies that $\tilde k=k$, and it is well known (see e.g.~\cite[Example 2.4]{OV}) that $\abn_{\H^k}=\{0\}$ has codimension $2k+1$.
\end{example}

On the other hand, this is not always the case. The estimate on $\dim E_{k,W}$ we found in Proposition~\ref{prop_EkW} may be loose due to two possible issues:
\begin{itemize}
\item
{\em First issue:} we may have $\supp(\omega_1)=\supp(\omega_2)$ for two linearly independent $2$-vectors $\omega_1,\omega_2 \in W^{\perp_2} \setminus \set{0}$, as in Example~\ref{exam_1} below. 

\item 

{\em Second issue:} we  computed the dimension of the algebraic variety defined by
\begin{equation*}
P_1(\xi,a,b)=\dots=P_{n-m}(\xi,a,b)=0,
\end{equation*}
while the variety defined by
\begin{equation*}
P_1(\xi,a,b)=\dots=P_{n}(\xi,a,b)=0
\end{equation*}
may have a lower dimension, despite not being a smooth manifold around $\omega_0$. Observe that this does not constitute a problem when $\mathcal{A}_{k,W}^0$ is non-empty, while it is the reason behind the following Example~\ref{exam_2}.
\end{itemize}

We now provide two examples of  Carnot groups whose abnormal sets has even codimension: therefore, in both cases the estimate provided by Theorem~\ref{thm_step2} is not optimal.

\begin{example}\label{exam_1}
Let $\G$ be the 6-dimensional  Carnot group of step 2 whose stratified algebra $\g=\g_1\oplus\g_2$ is such that
\begin{align*}
& \g_1=\spn\set{X_1,X_2,X_3,X_4},\qquad \g_2=\spn\set{T_1,T_2}
\end{align*}
and the only non-vanishing commutation relations between the generators are given by
\[
[X_1,X_2]=[X_3,X_4]=T_1,\qquad [X_1,X_4]=[X_2,X_3]=T_2.
\]
We can see $\g$ as the quotient $\mfk f_4/W$ of the free algebra $\mfk f_4=V\oplus\bigwedge^2 V$ (where $V:=\g_1$ ) by the subspace $W\le \bigwedge^2 V$ defined by
\begin{align*}
& W^{\perp_2} := \spn \set{\omega_1,\omega_2},\qquad\text{where}\\
& \omega_1:=X_1 \wedge X_2 + X_3 \wedge X_4\qquad\text{and}\qquad\omega_2:=X_1 \wedge X_4 + X_2 \wedge X_3.
\end{align*}
The adequate scalar product on $\f_4$ we consider is the one such that $X_1,X_2,X_3,X_4$ is orthonormal. We claim that $W^{\perp_2}$ does not contain simple $2$-vectors. Indeed,  the determinant of the skew-symmetric matrix $A_{t\omega_1+s\omega_2}$ associated with the linear combination $t\omega_1+s\omega_2$ is
	\begin{equation*}
		\det(A_{t\omega_1+s\omega_2})= \begin{vmatrix}
			0 & t & 0 & s \\ -t & 0 & s & 0 \\ 0 & -s & 0 & t \\ -s & 0 & -t & 0 
		\end{vmatrix} = (t^2 + s^2)^2,
		\end{equation*}
which is not null as long as $t$ and $s$ are not both zero. It follows that every non trivial linear combination $\omega$ of $\omega_1$ and $\omega_2$ has rank $2$ and, in particular, $\supp\;\omega=V$.	
	Recalling~\eqref{eq_abnG_EkW}
	\begin{align*}
		\abn_\G &= \bigcup \Set{\pi \left(\supp(\omega)^{\perp_1} \oplus \textstyle\bigwedge^2\supp(\omega)^{\perp_1}\right) \, | \, \omega \in W^{\perp_2} \setminus \set{0}, \, \rank(\omega)=2 \,} 
		=\set{0},
	\end{align*} 
	hence $\abn_\G=\{0\}$ has codimension $6$.
\end{example}

In the previous example, the codimension is less than $5$ (as estimated by Theorem~\ref{thm_step2}) because of the first issue, as $\omega_1$ and $\omega_2$ are two linearly independent $2$-vectors in $W^{\perp_2}$ with the same support. 
In order to construct an example where the codimension is $4$, we will instead build on the second issue. Namely, we will provide an example where $\mathcal{A}_{1,W}\ne\emptyset$, but $\mathcal{A}_{1,W}^0=\emptyset$, taking advantage of the geometry of simple 2-vectors.

\begin{example}\label{exam_2}
Let $\G$ be the 6-dimensional  Carnot group of step 2 whose stratified algebra $\g=\g_1\oplus\g_2$ is such that
\begin{align*}
& \g_1=\spn\set{X_1,X_2,X_3,X_4},\qquad \g_2=\spn\set{T_1,T_2}
\end{align*}
and the only non-vanishing commutation relations between the generators are given by
\[
[X_1,X_2]=T_1,\qquad [X_1,X_4]=[X_2,X_3]=T_2.
\]
We can see $\g$ as the quotient $\mfk f_4/W$ of the free algebra $\mfk f_4=V\oplus\bigwedge^2 V$ (where $V:=\g_1$ ) by the subspace $W\le \bigwedge^2 V$ defined by
\begin{align*}
& W^{\perp_2}:= \spn \set{\omega_1,\omega_2},\qquad\text{where}\\
& \omega_1:=X_1 \wedge X_2 \qquad\text{and}\qquad\omega_2:=X_1 \wedge X_4 + X_2 \wedge X_3.
\end{align*}
The adequate scalar product on $\f_4$ we consider is the one such that $X_1,X_2,X_3,X_4$ is orthonormal. The determinant of the skew-symmetric matrix $A_{t\omega_1+s\omega_2}$ associated with the linear combination $t\omega_1+s\omega_2$ is
	\begin{equation*}
		\det(A_{t\omega_1+s\omega_2})= \begin{vmatrix}
			0 & t & 0 & s \\ -t & 0 & s & 0 \\ 0 & -s & 0 & 0 \\ -s & 0 & 0 & 0 
		\end{vmatrix} = s^4,
		\end{equation*}
which is zero if and only if $s=0$; in particular, a linear combination $\omega=t\omega_1 + s\omega_2=t\omega_1$ ha rank 1 if and only if it is a (non-zero) multiple of $X_1\wedge X_2$. 
Recalling~\eqref{eq_abnG_EkW}
	\begin{align*}
		\abn_\G &= \pi \left(\supp(X_1 \wedge X_2)^{\perp_1} \oplus \textstyle\bigwedge^2\supp(X_1 \wedge X_2)^{\perp_1}\right) \\ 
		& \quad \cup \bigcup \Set{\pi \left(\supp(\omega)^{\perp_1} \oplus \textstyle\bigwedge^2\supp(\omega)^{\perp_1}\right) \, | \, \omega \in W^{\perp_2} \setminus \set{0}, \, \rank(\omega)=2 \,} \\
		&= \spn \set{X_3,X_4} \oplus \set{0}\   \cup\ \set{0} \,, 
	\end{align*} 
where we used the fact that $\bigwedge^2 \spn \set{X_3,X_4} \subseteq W$. Eventually, $\abn_\G=\spn \set{X_3,X_4}$ has codimension $4$.
\end{example}

\section{The Sard problem in filiform groups}
	\subsection{Filiform groups}\label{subsec_filiformi}
A {\em filiform group} is a Carnot group 	associated with a stratified Lie algebra $\galg=\galg_1\oplus\dots\oplus\galg_s$ such that
\[
\dim\galg_1=2\qquad\text{and}\qquad\dim\galg_2=\dots=\dim\galg_s=1.
\]
We fix a basis $X_1,\dots,X_{s+1}$ of $\galg$ such that
\[
\galg_1=\Span\{X_1,X_2\}\qquad\text{and}\qquad\galg_j=\Span\{X_{j+1}\}\ \forall\:j=2,\dots,s.
\]
There are  two non-isomorphic classes of filiform groups only (see~\cite{Vergne}), which we list according to their non-trivial bracket relations:
	\begin{itemize}
		\item Type I filiform groups, where  the only non-trivial relations are given by
		\[
			\begin{aligned}
				& X_3 = [X_1,X_2],\\
				& X_4 = [X_1,X_3]=[X_1,[X_1,X_2]],\\
				& \vdots \\
				& X_{s+1} = [X_1,X_{s}]=\underbrace{[X_1,[\dots,[X_1}_{\text{$(s-1)$-times}},X_2]\dots]].
			\end{aligned}
		\]
		\item Type II filiform groups, where $s$ is odd and the only non-trivial relations are given by
		\[
			\begin{aligned}
				& X_3 = [X_1,X_2],\\
				& X_4 = [X_1,X_3]=[X_1,[X_1,X_2]],\\
				& \vdots \\
				& X_{s} = [X_1,X_{s-1}]=\underbrace{[X_1,[\dots,[X_1}_{\text{$(s-2)$-times}},X_2]\dots]],\\
				& X_{s+1} = (-1)^i[X_i, X_{s+2-i}], \ \ \ \ i=2,\dots,s.
			\end{aligned}
		\]
	\end{itemize}
	
\begin{remark}\label{rem_I=II_in_step_3}
We observe that the filiform groups of type I and II with nilpotency step $s=3$ are isomorphic (Engel group); we can therefore adopt the convention that the step of a type-II filiform is an odd integer $s\geq 5$.
\end{remark}
	
We will denote by $X_1^*,\dots, X_{s+1}^*$ the basis of $\galg^*$ dual to $X_1,\dots,X_r$; accordingly,  each $\lambda\in\galg^*$ is written in these coordinates as $\lambda=\lambda_1 X_1^*+\dots+\lambda_{s+1}X_{s+1}^*$ for suitable $\lambda_1,\dots,\lambda_{s+1}$.

\subsection{Type I filiform groups} 
Let us characterize singular curves in type I filiform groups of step $s\geq 3$. For the well-known case $s=2$ (i.e., the first Heisenberg group) see Remark~\ref{rem_Heis1}.

\begin{proposition}\label{prop_filiformiI}
Let $\G$ be a type I filiform group of step $s\geq 3$ with generators $X_1,\dots,X_{s+1}$, as in \S~\ref{subsec_filiformi}. Then singular controls $u\in L^1([0,1],\galg_1)$ are exactly those for which  $u_1=0$ a.e. on $[0,1]$. In particular,  abnormal curves  are the absolutely continuous curves contained in the line $t\mapsto \exp(tX_2)$ and $\abn_\G=\{\exp(t X_2)\,|\,t\in\R\}$ has codimension $s$.
\end{proposition}
\begin{proof}
Let $u\in L^1([0,1],\galg_1)$ be a singular control. Using~\eqref{eq_alt_char_im} with $Y=X_2$ and taking the bracket relations into account, Proposition~\ref{prop_prop11BV} implies that the subspace
\begin{equation}\label{eq_imfilI}
		\begin{aligned}
		\Span_{t\in [0,1]} & \left\{X_2+\int_0^t u_1(\tau_1)d\tau_1\:X_3+\iint\limits_{0\le\tau_2\le \tau_1\le t} u_1(\tau_2)u_1(\tau_1) d\tau_2 d\tau_1\:X_4\right. \\
		& \left.\quad +\dots+\idotsint\limits_{0\le\tau_{s-1}\le\dots\le \tau_1\le t}u_1(\tau_{s-1})\dots u_1(\tau_1)d\tau_{s-1}\dots d\tau_1\:X_{s+1}\right\}
		\end{aligned}
\end{equation}
is contained in $(d_eR_{\gamma_u(1)})^{-1}(\IM d_u F)$. Since $u$ is singular, there exists $\lambda\in \mfk{g}^*$ such that $\lambda\neq 0$ and $\lambda\perp(d_eR_{\gamma_u(1)})^{-1}(\IM d_u F)$. By~\eqref{eq_inv} we have
\begin{equation}\label{eq_lambda120}
\lambda_1=\lambda_2=0.
\end{equation}
Since $\lambda$ is orthogonal to all the elements in~\eqref{eq_imfilI} we deduce that for every $t\in [0,1]$
\begin{equation}\label{eq_covectorI}
\begin{aligned}
A_1^\lambda(t):=
&\lambda_{3}\int_0^t u_1(\tau_1)d\tau_1
+\lambda_4\iint\limits_{0\le\tau_2\le \tau_1\le t} u_1(\tau_2)u_1(\tau_1)d\tau_2 d\tau_1\\
&+\dots+\lambda_{s+1}\idotsint\limits_{0\le\tau_{s-1}\le\dots\le \tau_1\le t}u_1(\tau_{s-1})\dots u_1(\tau_1)d\tau_{s-1}\dots d\tau_1\ =\ 0.
\end{aligned}
\end{equation}
For $i=2,\dots, s-1$ we introduce $A_i:[0,1]\to\R$ by
\[
	\begin{aligned}
			A_i^\lambda(t)&:=\lambda_{i+1}+\lambda_{i+2}\int_0^t u_1(\tau_i)d\tau_i+\dots+\lambda_{s+1}\idotsint\limits_{0\le\tau_{s-1}\le\dots\le \tau_i\le t}u_1(\tau_{s-1})\dots u_1(\tau_i)d\tau_{s-1}\dots d\tau_i,\\
			A_s^\lambda(t)&:=\lambda_{s+1}
	\end{aligned}
\]
so that for each $i=1,\dots, s-1$
\begin{equation}\label{eq_derivateAit}
\frac{d}{dt}A_i^\lambda(t)=u_1(t)A_{i+1}^\lambda(t)\qquad\text{for a.e. }t\in[0,1].
\end{equation}

Assume that $C_0:= \left\{t\in [0,1]\mid u_1(t)\ne 0\right\}$ is such that $\mathscr{L}^1(C_0)>0$; then  
\[
	C_1:=\left\{ t\in C_0\mid t \text{ is a Lebesgue point of } u_1\right\}
\]
satisfies $\mathscr{L}^1(C_1)>0$. Observe that $C_1$ does not contain any isolated point.
Differentiating \eqref{eq_covectorI} and using~\eqref{eq_derivateAit} we find $u_1(t)A_2^\lambda(t)=0$ for a.e. $t\in[0,1]$, therefore
\[
A_2^\lambda(t)=0\qquad \text{for a.e. } t\in C_1.
\]
We claim that also
\[
A_3^\lambda(t)=0\qquad \text{for a.e. } t\in C_1.
\]
In fact, by~\eqref{eq_derivateAit} one has that at every differentiability point $t\in C_1$ of $A_2^\lambda$ 
\[
	A_2^\lambda(t+h)=h \cdot u_1(t)  A_3^\lambda(t)+o(h)\ne 0
\]
for all $h$ sufficiently small, implying that $t+h\not\in C_1$ for all sufficiently small values of $h$, i.e. that $t$ is isolated in $C_1$. This would be a contradiction.

Inductively, suppose that we have proved that $A_i^\lambda(t)$ is zero for a.e. $t\in C_1$ and let us then show that $A_{i+1}^\lambda(t)=0$ for a.e. $t\in C_1$. If, by contradiction, this were not true, then at differentiability points of $A_i^\lambda$ we would have
\[
	A_i^\lambda(t+h)=h \cdot u_1(t)  A_{i+1}^\lambda(t)+o(h)\ne 0
\]  
for all $h$ sufficiently small, implying that $t+h\not\in C_1$ for all such values of $h$, contradicting the fact that $C_1$ does not contain any isolated point.

We have thus proved the following: almost every $t\in C_1$ is a common zero for the functions $A_1,\dots, A_s$. But this readily implies that $\lambda_{s+1}=\dots=\lambda_{3}=0$, which is impossible since the covector $\lambda$ has to be nonzero. 

We conclude that a singular control $u$ satisfies $u_1(t)=0$ a.e. on $[0,1]$. Conversely, every control $u$ such that $u_1\equiv 0$ is indeed singular: in fact, in this case Proposition~\ref{prop_prop11BV} gives
\[
\IM(d_u F)= d_eR_{\gamma_u(1)}(\galg_1\oplus\galg_2)
\]
unless  $u_2=0$ on $[0,1]$ as well, in which case $ \IM(d_u F)= d_eR_{\gamma_u(1)}(\galg_1) $. In both cases, we deduce that $u$ is singular because $s\geq 3$, and this concludes the proof. 
%
%
\end{proof}

\begin{remark}\label{rem_Heis1}
When $s=2$ the filiform group $\G$ is the 3-dimensional Heisenberg group $\H^1$. In this setting one can follow the previous proof to find that the only singular control is the null one. In particular, $\abn_{\H^1}=\{e\}$.
\end{remark}

\subsection{Type II filiform groups} 
We now study singular curve in type II filiform groups; recall that the nilpotency step of such groups is an odd integer not smaller than 5, see Remark~\ref{rem_I=II_in_step_3}.

\begin{proposition}\label{prop_filiformiII}
If $\G$ is a type II filiform group of step $s\geq 5$,
then $u\in L^1([0,1],\galg_1)$ is a singular control if and only if there exists $a\in\R$ such that both the following statements
\begin{itemize}
\item[(i)] either $u_1(t)=0$ or $u_2(t)=0$
\item[(ii)] if $u_1(t)\neq 0$, then $\int_0^t u_2(\tau)\:d\tau=a$
\end{itemize}
hold for a.e. $t\in[0,1]$.
\end{proposition}

Before proving Proposition~\ref{prop_filiformiII}, let us discuss its implications about the geometry of  abnormal curves in type II filiform groups.

\begin{remark}\label{rem_struttura_abnormali_filiformiII}
Let us briefly discuss the geometry of the abnormal curves associated with the singular controls described in Proposition~\ref{prop_filiformiII}. Let $u\in L^1([0,1],\galg_1)$ be  a control as in Proposition~\ref{prop_filiformiII}  and let $X:[0,1]\to\galg_1$ be a primitive of $u$. Condition (i) implies that $X$ is a concatenation of segments parallel to the coordinate axis in $\galg_1\equiv\R^2$, while condition (ii) requires that  the segments that are parallel to the first axis are all contained in the line $X_2=a$. See Figure~\ref{figuraabnII}. The abnormal curve $\gamma_u$ is uniquely determined by $X(t)$.

\begin{figure}[h]\label{figuraabnII}
	\centering
	\includegraphics[scale=1.5]{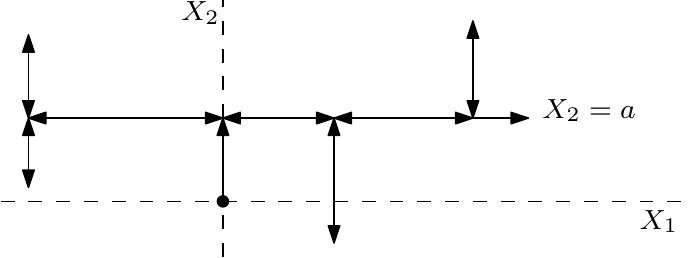}
	\caption{Example of an  abnormal curve in type II filiform groups. 
	}
\end{figure}

It is a standard task to deduce that 
\[
\abn_\G=\{\exp(a X_2)\exp(bX_1)\exp(cX_2)\mid a,b,c\in\R\}
\]
and, in particular, $\abn_\G$ is an algebraic  variety of dimension 3. In particular, $\abn_\G$ has codimension $s-2\geq3$.
\end{remark}

%

We  now prove Proposition~\ref{prop_filiformiII}.

\begin{proof}[Proof of Proposition~\ref{prop_filiformiII}]
{\em Step 1}.
Let $u\in L^1([0,1],\galg_1)$ be a singular control. Using~\eqref{eq_alt_char_im} with $Y=X_2$ and taking the bracket relations into account, Proposition~\ref{prop_prop11BV} implies that the subspace
\begin{equation}\label{eq_imfilII}
		\begin{aligned}
		\Span_{t\in [0,1]} & \left\{X_2+\int_0^t u_1(\tau_1)d\tau_1\:X_3+\iint\limits_{0\le\tau_2\le \tau_1\le t} u_1(\tau_2)u_1(\tau_1) d\tau_2 d\tau_1\:X_4\right. \\
		& \left.\quad +\dots+\idotsint\limits_{0\le\tau_{s-1}\le\dots\le \tau_1\le t}u_2(\tau_{s-1})u_1(\tau_{s-2})\dots u_1(\tau_1)d\tau_{s-1}\dots d\tau_1\:X_{s+1}\right\}
		\end{aligned}
\end{equation}
is contained in $(d_eR_{\gamma_u(1)})^{-1}(\IM d_u F)$. Since $u$ is singular, there exists $\lambda\in \mfk{g}^*$ such that $\lambda\neq 0$ and $\lambda\perp(d_eR_{\gamma_u(1)})^{-1}(\IM d_u F)$. By~\eqref{eq_inv} we have
\begin{equation}\label{eq_lambda120II}
\lambda_1=\lambda_2=0.
\end{equation}
Since $\lambda$ is orthogonal to all the elements in~\eqref{eq_imfilII} we deduce that for every $t\in [0,1]$
\begin{equation}\label{eq_covectorII}
\begin{aligned}
A_1^\lambda(t)\stackrel{\bigstar}{:=}
&\lambda_{3}\int_0^t u_1(\tau_1)d\tau_1
+\lambda_4\iint\limits_{0\le\tau_2\le \tau_1\le t} u_1(\tau_2)u_1(\tau_1)d\tau_2 d\tau_1\\
&+\dots+\lambda_{s+1}\idotsint\limits_{0\le\tau_{s-1}\le\dots\le \tau_1\le t}u_2(\tau_{s-1})u_1(\tau_{s-2})\dots u_1(\tau_1)d\tau_{s-1}\dots d\tau_1\ =\ 0.
\end{aligned}
\end{equation}
For $i=2,\dots, s-2$ we introduce $A_i^\lambda:[0,1]\to\R$ by
\begin{equation}\label{eq:functionsII}
	\begin{aligned}
			A_i^\lambda(t)&:=\lambda_{i+1}+\lambda_{i+2}\int_0^t u_1(\tau_i)d\tau_i+\dots+\lambda_{s+1}\idotsint\limits_{0\le\tau_{s-1}\le\dots\le \tau_i\le t}u_2(\tau_{s-1})\dots u_1(\tau_i)d\tau_{s-1}\dots d\tau_i,\\
			A_{s-1}^\lambda(t)&:=\lambda_s+\lambda_{s+1}\int_0^t u_2(\tau)d\tau\\
			A_s^\lambda(t)&:=\lambda_{s+1},
	\end{aligned}
\end{equation}
so that for each $i=1,\dots, s-1$
\begin{equation}\label{eq_derivateAitII}
\frac{d}{dt}A_i^\lambda(t)
=\left\{ 
\begin{array}{ll} 
u_1(t)A_{i+1}^\lambda(t),\quad &i=1,\dots, s-2,\vspace{.1cm}\\
u_2(t)A_s^\lambda(t)=u_2(t)\lambda_{s+1},\qquad &i=s-1
\end{array}\right.
\qquad\text{for a.e. }t\in [0,1].
\end{equation}

We first prove statement (i), i.e.,  that $u_1u_2=0$ a.e. on $[0,1]$. 
Assume by contradiction that $C_0:= \left\{t\in [0,1]\mid u_1(t)u_2(t)\neq0\right\}$ is such that $\mathscr{L}^1(C_0)>0$; then also 
\[
	C_1:=\left\{ t\in C_0\mid t \text{ is a Lebesgue point of } u\right\}
\]
satisfies $\mathscr{L}^1(C_1)>0$. Again, $C_1$ does not contain any isolated point and, arguing as in the proof of Proposition~\ref{prop_filiformiI}, we deduce that $A_1^\lambda(t)=\dots=A_s^\lambda(t)=0$ for almost every $t\in C_1$. This  gives
\[
	\lambda_{s+1}=\lambda_{s}=\dots=\lambda_{3}=0,
\]
contradicting the fact that $\lambda$ is non-zero.

Next, we prove that, if $\lambda_{s+1}=0$, then necessarily $u_1=0$ a.e. on $[0,1]$, so that also statement (ii) holds. Assume on the contrary that the set $D_0:=\{t\in [0,1]\mid u_1(t)\ne 0\}$ has positive measure; then also the set $D_1\subset D_0$ of Lebesgue points of $u$ has positive measure and (using in a crucial way the assumption $\lambda_{s+1}=0$) one can reason as before to deduce that $\lambda=0$, contradiction.

Eventually, we consider the case $\lambda_{s+1}\neq 0$; we can also assume that the sets $D_0$ and $D_1$ introduced in the previous paragraph have positive measure, otherwise we have again $u_1=0$ a.e. on $[0,1]$.
Differentiating~\eqref{eq_covectorII} and using~\eqref{eq_derivateAitII} we deduce that there exists $D_2\subset D_1$ such that $\mathscr L^1(D_1\setminus D_2)=\mathscr L^1(D_0\setminus D_2)=0$ and
\begin{equation}\label{eq_puntidiD1II}
A_1^\lambda(t)=A_2^\lambda(t)=\dots=A_{s-1}^\lambda(t)=0\qquad\forall\; t\in D_2.
\end{equation}
Therefore
\[
0=A_{s-1}^\lambda(t)=\lambda_{s}+\lambda_{s+1}\int_0^{t}u_2(\tau)d\tau\qquad\forall\; t\in D_2,
\]
and we get that necessarily
\[
\int_0^{t}u_2(\tau)d\tau
=
-\frac{\lambda_{s}}{\lambda_{s+1}}
=:
a\in \R\qquad\forall\; t\in D_2.
\]
Since $\mathscr L^1(D_0\setminus D_2)=0$, statement (ii) is proved for a.e. $t$.

{\em Step 2.}
Conversely, let $a\in\R$ and a control $u\in L^1([0,1],\galg_1)$ be fixed so that statements (i) and (ii) hold for a.e. $t\in [0,1]$. Let $\lambda\in\galg^*$ be defined by
\begin{equation}\label{eq_deflambda}
\lambda_1=\dots=\lambda_{s-1}=0,\qquad\lambda_s=-a,\qquad\lambda_{s+1}=1;
\end{equation}
we consider the functions $A_1^\lambda,\dots, A_s^\lambda:[0,1]\to\R$ defined by $\bigstar$ in~\eqref{eq_covectorII} and by~\eqref{eq:functionsII}, and we introduce $B_1^\lambda:[0,1]\to\R$ defined by
\[
\begin{aligned}
B_1^\lambda(t):=
&\lambda_{s}\idotsint\limits_{0\le\tau_{s-2}\le\dots\le \tau_1\le t}
u_1(\tau_{s-2})\dots u_1(\tau_{2})u_2(\tau_1)d\tau_{s-2}\dots d\tau_1
\\
&+\lambda_{s+1}\idotsint\limits_{0\le\tau_{s-1}\le\dots\le \tau_1\le t}u_1(\tau_{s-1})\dots u_1(\tau_{2}) u_2(\tau_1)d\tau_{s-1}\dots d\tau_1.
\end{aligned}
\]
Let $Y=x_1X_1+x_2X_2\in\galg_1$ and $t\in[0,1]$ be fixed.
Taking also~\eqref{eq_deflambda} into account, we see that that pairing of $\lambda$ with the vector
\[
Y+\sum_{j=1}^{s-1}\int_{\Sigma_j(t)}\left(\mathrm{ad}X_{u(\tau_{j})}\circ\dots\circ  \mathrm{ad}X_{u(\tau_1)}  \right)Y\;d\tau_{j}\dots d\tau_1 
\]
is equal to $x_1A_1^\lambda(t)+x_2B_1^\lambda(t)$. We claim that
\begin{equation}\label{eq_A1B1nulli}
A_1^\lambda=B_1^\lambda=0\qquad\text{on }[0,1];
\end{equation}
thanks to Proposition~\ref{prop_prop11BV}, this implies that $\lambda$ is orthogonal to $(d_eR_{\gamma_u(1)})^{-1}(\IM d_u F)$, which is therefore a proper subspace of $\galg$ ensuring  that  $u$ is a singular control.

Let us prove~\eqref{eq_A1B1nulli}. We observe that~\eqref{eq_derivateAitII} still hold, while~\eqref{eq_deflambda} together with statement  (ii) implies that for a.e. $t\in[0,1]$
\[
\text{either}\qquad
u_1(t)=0\qquad
\text{or}\qquad
A_{s-1}^\lambda(t)=-a+\int_0^tu_2(\tau)d\tau=0.
\]
By~\eqref{eq_derivateAitII} this gives $\frac{d}{dt}A_{s-2}^\lambda=0$, hence $A_{s-2}^\lambda$ is constant and $A_{s-2}^\lambda=A_{s-2}^\lambda(0)=\lambda_{s-1}=0$. This argument can be repeated to prove that $A_{s-3}^\lambda,\dots,A_1^\lambda$ are constant, and actually identically zero due to~\eqref{eq_deflambda}. Eventually, we observe that $B_1^\lambda(0)=0$ and $\frac{d}{dt}B_{1}^\lambda=u_2A_2^\lambda=0$ a.e. on $[0,1]$, so also $B_1^\lambda$ is identically zero in $[0,1]$. Our claim~\eqref{eq_A1B1nulli} is proved and the proof is concluded.
\end{proof}

\begin{remark}
It is worth noticing that, as soon as the singular control $u$ is not such that $u_1=0$ a.e. on $[0,1]$, then the covector $\lambda$ associated with  $u$ (as in Step 1 of the proof of Proposition~\ref{prop_filiformiII}) satisfies
\[
\lambda_1=\lambda_2=\dots=\lambda_{s-1}=0.
\]
Indeed the set $D_2$ introduced in the proof of Proposition~\ref{prop_filiformiII} has positive measure and~\eqref{eq_puntidiD1II} holds; we can fix a sequence of points $(t_k)_{k\in \N}\subset D_2$ converging to $t_0:=\inf D_2$.
Since the functions $A_{i}^\lambda$ are continuous, the equalities~\eqref{eq_puntidiD1II} hold for $t=t_0$ as well and, using the fact that $u_1=0$ a.e. in $ [0,t_0)$ we obtain
\[
0=A_{s-2}^\lambda(t_0)=\lambda_{s-1}.
\]
Inductively, this argument can be used to show that
\[
	\lambda_{s-1}=\dots=\lambda_{3}=0,
\]
which is enough to conclude thanks to~\eqref{eq_lambda120II}.
\end{remark}

\bibliographystyle{acm}
\bibliography{biblio}

\end{document}